\documentclass[10pt]{article}
\usepackage{amsmath}
\usepackage{amssymb}
\usepackage{amsthm}
\usepackage{gensymb}
\usepackage[margin=1.34in]{geometry}
\usepackage[utf8]{inputenc}
\usepackage{tikz-cd}
\usepackage{graphicx}
\usepackage{color,soul}
\newenvironment{rcases}
  {\left.\begin{aligned}}
  {\end{aligned}\right\rbrace}

\numberwithin{equation}{section}

\newtheorem{thm}{Theorem}[section]

\newcommand{\RNum}[1]{\uppercase\expandafter{\romannumeral #1\relax}}

\begin{document}
\title{On the  initial value problem for the  electromagnetic wave equation in Friedmann-Robertson-Walker space-times}
\author{Walter Craig and Mikale Reddy
\\ \\
Department of Mathematics and Statistics, McMaster  University\\
Hamilton,  Ontario, L8S 4K1, Canada\\
reddymikale@gmail.com}
\date{\today}

\maketitle

\begin{abstract}
We solve the source free electromagnetic wave equation in Friedmann-Robertson-Walker space-times for curvature $K=0$ and $K=-1$. Deriving a solution expression in the form of spherical means we deduce and compare two properties of the Maxwell propagator namely, decay rates, as well as continuity through the space-time singularity to that of the scalar wave equation presented by Abbasi and Craig [1].
\end{abstract}

\section{Introduction}
 Friedmann-Robertson-Walker (FRW) metrics are Lorentzian metrics  which describe an expanding (or contracting) space-time satisfying the symmetries of spatial homogeneity and isotropy. In addition, they are metrics which possess a space-time singularity at time $t=0$. These metrics play a central role in general relativity and cosmology as they're the simplest models of the universe which contain a `big bang' singularity.\\\\
In a recent paper, Abbasi and Craig [1] derived three results about the scalar wave propagator for the Cauchy problem in these space-times. In this paper we derive and compare analogous results for the source free electromagnetic wave equation, of which is a vector wave equation, to that of the scalar wave equation presented in [1]. First, we derive decay rates of the solution with compactly supported initial data. Secondly, noting that all FRW spaces are conformally flat, Huygens' principle is obeyed for all cases considered as they are obeyed in the underlying product metrics considered (Helgason S.) [6].  Lastly, we show that for $0<t_0<t$ the limit of the wave propagator,
\[\lim_{t_0\to0^+}W(t_0,t)(f^\mu,g^\mu)=W(0,t)(f^\mu,g^\mu)\]
exists and gives rise to a well-defined solution for all $t>0$ depending upon the initial data $(f^\mu(x),g^\mu(x))$ emanating from the space-time singularity at $t=0$. Under the reflection $t\to -t$, the FRW metric gives a space-time metric for $t<0$ with a singular future at $t=0$. Due to the symmetry of the solution expressed in terms of spherical averages the same solution formulae hold. Thus, analogous to the results obtained by Abbasi and Craig [1] we have constructed solutions $A^\mu(t,x)$ of the electromagnetic wave equation in Friedmann-Robertson-Walker space-times which exist for all $-\infty<t<0$ and $0<t<+\infty$ where in conformal coordinates, are continuous through the singularity at  $t=0$, taking on specified data $A^\mu(0,\cdot)=(f^\mu(\cdot),g^\mu(\cdot))$.\\\\
In FRW space-time, the general metric takes the form
\begin{equation}
ds^2=-dt^2+a^2(t)d\sigma^2
\end{equation}
where $a^2(\tau)$ is  the scale factor and $d\sigma^2$ is the line element for each spatially homogeneous time slice. These correspond to Euclidean space $\mathbb{R}^3$, and hyperbolic space $\mathbb{H}^3$ for curvature $K=0$ and $K=-1$ respectively in the cases considered here. Under the coordinate transformation
\[\frac{dt}{d\tau}=a(t)\]
i.e.
\[\tau=\int\frac{dt}{a(t)}.\]
the metric (1.1) takes the form [1]
\begin{equation}
ds^2=a^2(\tau)(-d\tau^2+d\sigma^2).
\end{equation}
In  this form, it becomes evident that the metric is a conformal change of the underlying product metric. In addition, it can be  shown that under this coordinate transformation the FRW space-time is conformal to the upper half space $\mathbb R_+\times \mathbb R^3=\{(t,x):\tau>0\}$ with the Minkowski metric and time variable $\tau$. In terms of the transformed time variable $\tau$ the scale factor $a(\tau)$ takes the form (Ellis R., Hawking S.) [4]
\begin{align*}
\begin{rcases}
&a(\tau)=\tau^2, \qquad K=0 \\
&a(\tau)=\cosh(\tau)-1, \qquad K=-1\\
\end{rcases}
\end{align*}
where the conformal  time $\tau$ is related to the original time variable $t$ by
\begin{align*}
\begin{rcases}
&t=\frac{\tau^3}{3}, \qquad K=0 \\
&t=\sinh(\tau)-\tau, \qquad K=-1.\\
\end{rcases}
\end{align*}
\section{Conformal Invariance}
It is a well known result that Maxwell's equations are conformally invariant. Indeed consider the source free  action for the electromagnetic field tensor given by
\[S=\int\Big(-\frac{1}{4}F_{\mu\nu}F^{\mu\nu}\Big)\sqrt{-g}d^nx.\]
Under the conformal transformation $g(x)\to \lambda(x) g(x)=\tilde{g}(x)$, the electromagnetic field tensor $F^{\mu\nu}$ satisfies (Craig W., Starko D.) [3]
\[F^{\mu\nu}\to F_{\alpha\beta}g^{\alpha\mu}g^{\beta\nu}\to \lambda^{-2}(x)F_{\alpha\beta}\tilde{g}^{\alpha\mu}\tilde{g}^{\beta\nu}.\]
In addition the volume form transforms as
\[\sqrt{-g}d^nx\to\lambda^{\frac{n}{2}}(x)\sqrt{-\tilde{g}}d^nx.\]
It  follows that the action $S$ is invariant under conformal transformations for $n=4$, implying that in the case of FRW background metrics, the Maxwell propagator satisfies the same properties in an  FRW metric as it does in the underlying product metric. 
\section{Electromagnetic wave equation for $K=0$}
Consider a $\textit{flat}$ FRW space-time  whose line element in conformal time $\tau$ is given as
\[ds^2=a^2(\tau)(-d\tau^2+dx^2+dy^2+dz^2)\]
conformal to the upper half space $\mathbb R_+\times \mathbb R^3=\{(t,x):\tau>0\}.$\\\\
Varying the action we obtain one of Maxwell's equations in FRW  space-time for curvature $K=0$
\begin{equation}
-\frac{1}{4\pi}\partial_\nu\left(g^{\mu\mu}g^{\nu\nu}F_{\mu\nu}\right)=-\frac{1}{4\pi}\partial_\nu\left(F^{\mu\nu}\right)=0.
\end{equation}
Note that since the conformal factor drops out of the action, $g^{\mu\mu}=g^{\nu\nu}=\pm 1$ (the underlying product matrix elements).
The other can be realized by the Bianchi identities. Substituting in the condition
\[F^{\mu\nu}=\partial^\mu A^\nu-\partial^\nu A^\mu\]
where $A_\mu$ is the four potential, we obtain
\begin{equation}
-\frac{1}{4\pi}\partial_\nu\Big(\partial^\mu (g^{\nu\nu}A_\nu)\Big)+\frac{1}{4\pi}\partial_\nu\Big({\partial^\nu}(g^{\mu\mu}A_\mu)\Big)=0.
\end{equation}
We choose the four potential to be in the Coulomb gauge
\[\frac{1}{\sqrt{-g}}\partial_\nu\Big(\sqrt{-g}(g^{\nu\nu}A_\nu)\Big)=\frac{1}{a^4(\tau)}\Bigg(\partial_\nu\Big(a^4(\tau)(g^{\nu\nu}A_\nu)\Big)\Bigg)=0\]
where $\nu=1,2,3,$
\[A_0=0\]
 since this gauge is conformally invariant under conformal transformations of the space time metric with conformal factors of the form $a^2(\tau)$ (Alertz, B. section \RNum{7}) [2]. Upon this choice of gauge we obtain the wave equation for the spatial components of the four potential $A_\mu$ ($\mu=1,2,3$) in \textit{flat} FRW space-time
\begin{equation}
-\frac{1}{4\pi}\partial_\nu\Big({\partial^\nu}(g^{\mu\mu}A_\mu)\Big)=0
\end{equation} 
\subsection{The solution operator}
The $\textit{initial value problem}$ for equation (3.3) takes the form
\begin{equation}
\begin{rcases}
&\partial^2_\tau A^\mu-\Delta A^\mu=0, \quad  \tau>0, x\in \mathbb{R}^3\\
&A^\mu(\tau_0,x)=f^\mu(x)\\
&\partial^\tau A^\mu(\tau_0,x)=g^\mu(x)\\
\end{rcases}
\end{equation}
where $\Delta$  is the ordinary Laplace operator on $\mathbb{R}^3$. We assume $f^\mu(x)\in C^2(\mathbb{R}^3)$ and $g^\mu(x)\in C^1(\mathbb{R}^3)$  along with the condition that $\textit{supp}(f^\mu(x),g^\mu(x))\subseteq B_R(0)$ for some $R>0$ where $B_R(0)=\{x\in\mathbb{R}^3:|x|\leq R\}$. Define the spherical means operator to be 
\[M_f(r,x):=\frac{1}{4\pi r^2}\int_{S_r(x)}f^\mu(y)dS_r(y)\]
whose center lies at $x$ and radius $r$. The spherical means representation is given by Kirchoff's formula
\begin{equation}
A^\mu(\tau,x)=\partial_\tau((\tau-\tau_0)M_f(\tau-\tau_0,x))+(\tau-\tau_0)M_g(\tau-\tau_0,x).
\end{equation}
Carrying out the differentiation, the solution takes the explicit form
\begin{align}
A^\mu(\tau,x)&=\frac{1}{4\pi\left(\tau-\tau_0\right)^2}\bigg(\int_{S_{\tau-\tau_0}(x)}(\tau-\tau_0)g^\mu(y)dS_{\tau-\tau_0}(y)\nonumber\\&\quad+
\int_{S_{\tau-\tau_0}(x)}f^\mu(y)dS_{\tau-\tau_0}(y)\nonumber\\ &\quad+
\int_{S_{\tau-\tau_0}(x)}(\tau-\tau_0)\left(\frac{x-y}{|x-y|}\cdot\nabla\right)f^\mu(y)dS_{\tau-\tau_0}(y)\bigg).
\end{align}
\subsection{Rate of decay}
\begin{thm}
Suppose $f^\mu\in C^2(\mathbb{R}^3)$ and $g^\mu\in C^1(\mathbb{R}^3)$, along with $supp(f^\mu,g^\mu)\subset B_R(0)$. Then the solution days to zero at a rate of $\mathcal{O}(\tau^{-1})$.
\end{thm}
\begin{proof}
Consider initial data $f^\mu\in C^2(\mathbb{R}^3)$ and $g^\mu\in C^1(\mathbb{R}^3)$, along with $supp(f^\mu,g^\mu)\subseteq B_R(0)$. Define constants
\[C_f:=sup_{x\in\mathbb{R}^3}|(f^\mu(x), \nabla f^\mu(x))|=|f^\mu(x)|_{C^2(B_R)}, \qquad C_g:=sup_{x\in\mathbb{R}^3}|g^\mu(x)|=|g^\mu(x)|_{C^1(B_R)},\]
Consider the third term in expression (3.6).
\[\frac{1}{4\pi\left(\tau-\tau_0\right)^2}\Bigg|\int_{S_{\tau-\tau_0}(x)}(\tau-\tau_0)\left(\frac{x-y}{|x-y|}\cdot\nabla\right)f^\mu(y)dS_{\tau-\tau_0}(y)\bigg)\Bigg|\]
\[\leq\frac{1}{(\tau-\tau_0)^2}\big|(\tau-\tau_0)\big|\frac{C_f}{4\pi}\int_{S_{\tau-\tau_0}(x)\cap B_R(0)}dS_{\tau-\tau_0}(y)\leq\frac{C_fR^2}{\tau-\tau_0}\]
Note that $|\int_{S_{\tau-\tau_0}(x)\cap B_R(0)}dS_{\tau-\tau_0}(y)|$ is bounded by $4\pi$ min$\{(\tau-\tau_0)^2,R^2\}$. Similar estimates hold for the remaining terms in (3.6). The results is  that for $\tau$ large, $|A^\mu(\tau,x)|\leq\mathcal{O}(\tau^{-1})$.
\end{proof}
Transforming back to the original time variable under the transformation $t=\tau^{3}/3$,  the decay rate is obtained as  $\mathcal O(t^{-1/3})$.   
\subsection{The initial value problem at the singular time $\tau_0=0$}
The Maxwell propagator for Maxwell's wave equation, $W(\tau_0,\tau_1)(f_\mu,g_\mu)$, is defined to be the solution operator
\[W(\tau_0,\tau_1)(f^\mu,g^\mu):=(A^\mu(\tau_1,x),\partial_\tau A^\mu(\tau_1,x)), \quad \tau_0>0, \tau_1>0\]
where $A^\mu$ is a solution to Maxwell's wave equation. Again we have worked under the condition that $\tau,\tau_0>0$ due to the space-time singularity at time $\tau=0$ (the image of $t=0$). However, due to the solution being given by such an explicit expression in the form of spherical means, we consider the limit
\[\lim_{\tau_0\to 0^+}W(\tau_0,\tau_1)(f^\mu,g^\mu)\]
where we take the Cauchy hypersurface defined at the initial time $\tau_0$ to zero while keeping $\tau_1$ fixed. We continue to work in the time variable $\tau$ as it is asymptotically the image of $t$ in a neighbourhood around the singularity  $t=0=\tau$. Using the initial data in (3.4) and taking the  limit as $\tau_0\to 0$ of expression (3.5) we obtain
\begin{equation}
A^\mu(\tau,x)=\partial_\tau((\tau M_f)(\tau,x))+\tau M_g(\tau,x)
\end{equation}
and upon carrying out the differentiation
\begin{align}
A^\mu(\tau,x)&=\frac{1}{4\pi\tau^2}\bigg(\int_{S_{\tau}(x)}(\tau)g^\mu(y)dS_{\tau}(y)\nonumber\\&\quad+
\int_{S_{\tau}(x)}f^\mu(y)dS_{\tau}(y)\nonumber\\ &\quad+
\int_{S_{\tau}(x)}(\tau)\left(\frac{x-y}{|x-y|}\cdot\nabla\right)f^\mu(y)dS_{\tau}(y)\bigg).
\end{align}
\begin{thm}
For $(f^\mu,g^\mu)\in C^2(\mathbb  R^3)\times C^1(\mathbb R^3)$ the limit of the wave propagator exists,
\begin{equation}
\lim_{\tau_0\to  0^+}  W(0,\tau)(f^\mu,g^\mu)=(A^\mu(\tau,x),\partial_\tau A^\mu(\tau,x)),
\end{equation}
it depends on $f^\mu(x)$ as well as $g^\mu(x)$, and  satisfies
\begin{equation}
\lim_{\tau\to 0^+} A^\mu(\tau,x)=f^\mu(x),  \quad \lim_{\tau\to 0+} \partial_\tau A^\mu(\tau,x)=g^\mu(x).
\end{equation}
Thus, the expression gives a solution to the wave equation over the full half-line\\ $\tau\in (0,+\infty)$, with initial data $(A^\mu(0,x),\partial_\tau A^\mu(0,x))=(f^\mu(x),g^\mu(x))$ given at the singular time  $\tau=0$.
\end{thm}
\begin{proof}
Consider expression (3.7). We find that
\[\lim_{\tau\to 0^+}A^\mu(\tau,x)=\lim_{\tau\to 0^+}M_{f}(\tau,x)=f^\mu(x)\]
for continuous initial data $f^\mu(x)$. In addition, after differentiating, it can be expressed as
\[\partial_\tau A^\mu(\tau,x)=2\partial_\tau M_f(\tau,x)+\tau\partial^2_\tau M_f(\tau,x)+M_g(\tau,x)+\tau\partial_\tau M_g(\tau,x)\]
from which we obtain as a limit
\[\lim_{\tau\to 0^+}\partial_\tau A^\mu(\tau,x)=\lim_{\tau\to 0^+}M_g(\tau,x)=g^\mu(x)\]
where the terms involving partial derivatives vanish since $M_f$ and $M_g$ are even in $\tau$.
\end{proof}
This calculation shows that the Cauchy problem is indeed well posed for $\tau_0=0$, where solutions of the Cauchy problem posed  at $\tau_0>0$ and propagated to times $0<\tau<\tau_0$ will attain there initial data on the Cauchy hypersurface as $\tau\to 0^+$. The solution itself exists for the initial value problem consisting of initial data $(f^\mu(x),g^\mu(x))$ posed at $\tau_0$, and propagated to future (past) times $\tau$.

\section{Electromagnetic wave equation $K=-1$}
Consider a $\textit{hyperbolic}$ FRW space-time with line element in conformal time $\tau$ given as
\[ds^2=a^2(\tau)\left(-d\tau^2+\frac{dx^2+dy^2+dz^2}{z^2}\right)\]
conformal to the upper half space $\mathbb R_+\times \mathbb R^3=\{(t,x):\tau>0\}.$
Varying  the action in this metric, we obtain as one of Maxwell's equations in FRW space-time for curvature $K=-1$
\begin{equation}
-\frac{1}{4\pi}z^3\partial_\nu\left(z^{-3}g^{\mu\mu}g^{\nu\nu}F_{\mu\nu}\right)=-\frac{1}{4\pi}z^3\partial_\nu\left(z^{-3}F^{\mu\nu}\right)=0.
\end{equation}
The other can once again be realized by the Bianchi identities. Substituting in the condition
\[F^{\mu\nu}=\partial^\mu A^\nu-\partial^\nu A^\mu\]
we obtain
\begin{equation}
-\frac{1}{4\pi}z^3\partial_\nu\Big(z^{-3}\partial^\mu (g^{\nu\nu}A_\nu)\Big)+\frac{1}{4\pi}z^3\partial_\nu\Big(z^{-3}{\partial^\nu}(g^{\mu\mu}A_\mu)\Big)=0.
\end{equation}
In order to obtain the wave equation in the case of curvature $K=-1$,  we once again impose the Coulomb gauge
\[\frac{1}{\sqrt{-g}}\partial_\nu\Big(\sqrt{-g}(g^{\nu\nu}A_\nu)\Big)=\frac{1}{a^4(\tau)z^{-3}}\Bigg(\partial_\nu\Big(a^4(\tau)z^{-3}(g^{\nu\nu}A_\nu)\Big )\Bigg)=0\]
where $\nu=1,2,3,$ 
\[A_0=0\]
as this gauge is conformally invariant under conformal transformations of the spacetime metric with conformal factors of the form $a^2(\tau)$ (Alertz, B. section \RNum{7}) [2]. Imposing these conditions, we obtain the wave equation for the spatial components of the four potential $A_\mu$ ($\mu=1,2,3$) in \textit{hyperbolic} FRW space-time
\begin{equation}
\frac{1}{4\pi}z^3\partial_\nu\Big(z^{-3}{\partial^\nu}(g^{\mu\mu}A_\mu)\Big)=0.
\end{equation}
Here we note that $g^{00}=-1$ and $g^{11}=g^{22}=g^{33}=z^{2}$.
\subsection{The solution operator}
The $\textit{initial value problem}$ for the four potential takes the form
\begin{equation}
\begin{rcases}
&\partial^2_\tau A^\mu-\Delta_\sigma A^\mu=0, \quad  \tau>0, x\in \mathbb{H}^3\\
&A_\mu(\tau_0,x)=f^\mu(x)\\
&\partial_\tau A^\mu(\tau_0,x)=g^\mu(x)\\
\end{rcases}
\end{equation}
where $\Delta_\sigma$  is the Laplace-Beltrami operator on $\mathbb{H}^3$. We assume $f^\mu(x)\in C^2(\mathbb{H}^3)$ and $g^\mu(x)\in C^1(\mathbb{H}^3)$  along with the condition that $\textit{supp}(f^\mu(x),g^\mu(x))\subseteq B_R(x_0)$ for some $R>0$ where $B_R(x_0)=\{x\in\mathbb{H}^3:|x|\leq R\}$. In addition we require that $B_R(x_0)\cap (x,y,0)=\emptyset$ in accordance with Maxwell's equations for the electric and magnetic fields. This system has an explicit spherical means formula for the solution given in (Klainerman S., Sarnak P.) [8], which is the hyperbolic analogue of the spherical mean expression given previously, namely
\begin{equation}
A^\mu(\tau,x)=\partial_\tau \left(\sinh(\tau-\tau_0)M_f(\tau-\tau_0,x)\right)+\sinh(\tau-\tau_0)M_g(\tau-\tau_0,x)
\end{equation}
where the geodesic spherical mean of a vector  $f^\mu(x)$  is given  by  the integral over the geodesic sphere $S_\tau(x)$ of radius $r$ centred at $x$
\[M_f(r,x):=\frac{1}{4\pi\left(\sinh(r)\right)^2}\int_{S_r(x)}f^\mu(y)dS_r(y)\]
where $dS_r(y)$ is the spherical surface area. We thus obtain as our solution formula via spherical means on hyperbolic FRW space-times;
\begin{align}
A^\mu(\tau,x)&=\frac{1}{4\pi\left(\sinh(\tau-\tau_0)\right)^2}\bigg(\int_{S_{\tau-\tau_0}(x)}\sinh(\tau-\tau_0)g^\mu(y)dS_{\tau-\tau_0}(y)\nonumber\\&\quad+
\int_{S_{\tau-\tau_0}(x)}\cosh(\tau-\tau_0)f^\mu(y)dS_{\tau-\tau_0}(y)\nonumber\\ &\quad+
\int_{S_{\tau-\tau_0}(x)}\sinh(\tau-\tau_0)\left(\frac{x-y}{|x-y|}\cdot\nabla\right)f^\mu(y)dS_{\tau-\tau_0}(y)\bigg).
\end{align}
\subsection{Rate of decay}
\begin{thm}
Suppose $f^\mu\in C^2(\mathbb{H}^3)$ and $g^\mu\in C^1(\mathbb{H}^3)$, along with $supp(f^\mu,g^\mu)\subset B_R(x_0)$. Then the solution decays to zero at a rate of $\mathcal{O}(e^{-\tau})$.
\end{thm}
\begin{proof}
Consider initial data $f^\mu\in C^2(\mathbb{H}^3)$ and $g^\mu\in C^1(\mathbb{H}^3)$, along with $supp(f^\mu,g^\mu)\subseteq B_R(x_0)$. Define constants
\[C_f:=sup_{x\in\mathbb{H}^3}|(f^\mu(x), \nabla f^\mu(x))|=|f^\mu(x)|_{C^2(B_R)}, \qquad C_g:=sup_{x\in\mathbb{H}^3}|g^\mu(x)|=|g^\mu(x)|_{C^1(B_R)},\]
Consider the third term in expression (4.6).
\[\frac{1}{4\pi\sinh(\tau-\tau_0)}\Bigg|\int_{S_{\tau-\tau_0}(x)}\left(\frac{x-y}{|x-y|}\cdot\nabla\right)f^\mu(y)dS_{\tau-\tau_0}(y)\bigg)\Bigg|\]
\[\leq\frac{1}{4\pi\sinh(\tau-\tau_0)}C_f\int_{S_{\tau-\tau_0}(x)\cap B_R(x_0)}dS_{\tau-\tau_0}(y)\leq\frac{C_f\sinh^2(R)}{\sinh(\tau-\tau_0)}\]
Note that $|\int_{S_{\tau-\tau_0}(x)\cap B_R(x_0)}dS_{\tau-\tau_0}(y)|$ is bounded by $4\pi$ min$\{\sinh^2(\tau-\tau_0),\sinh^2(R)\}$. Similar estimates hold for the other terms in (4.6). The result is that for large $\tau$, $|A^\mu(\tau,x)|\leq\mathcal{O}(e^{-\tau})$.
\end{proof}
Returning to the original time variable defined by the transformation $t=\sinh(\tau)-\tau$, the decay rate obtained from the asymptotics of the expression is given as $\mathcal{O}(t^{-1})$.
\subsection{The initial value problem at the singular time $\tau_0=0$}
We proceed  similarly as was done  for Maxwell's wave equation in FRW space-times with flat spatial geometry. Consider the Maxwell propagator $W(\tau_0,\tau_1)(f^\mu, g^\mu)$ defined to be
\[W(\tau_0,\tau_1)(f^\mu, g^\mu):=(A^\mu(\tau_1,x),\partial_\tau A^\mu(\tau_1,x)), \qquad \tau_0>0,  \tau_1>0\]
where $A^\mu$ is  a solution to Maxwell's wave equation. As in the case of spatially flat FRW space-times, we worked solely under the condition that  $\tau, \tau_0>0$ due to the space-time singularity at time $\tau=0$. However,  given the explicit nature of the wave propagator on $\mathbb{H}^3$ we can consider the limit
\[\lim_{\tau_0\to 0^{+}}W(\tau_0,\tau_1)(f^\mu,g^\mu)\]
where as previously we work in the time variable $\tau$, taking the Cauchy hypersurface defined at the initial time $\tau_0$ to zero while keeping $\tau_1$ fixed. From  this we obtain
\begin{equation}
A^\mu(\tau,x)=\partial_\tau\left(\sinh(\tau)M_f(\tau,x)\right)+\sinh(\tau)M_g(\tau,x)
\end{equation}
which has the explicit form
\begin{align}
A^\mu(\tau,x)&=\frac{1}{4\pi\left(\sinh(\tau)\right)^2}\bigg(\int_{S_{\tau}(x)}\sinh(\tau)g^\mu(y)dS_{\tau}(y)\nonumber\\&\quad+
\int_{S_{\tau}(x)}\cosh(\tau)f^\mu(y)dS_{\tau}(y)\nonumber\\ &\quad+
\int_{S_{\tau}(x)}\sinh(\tau)\left(\frac{x-y}{|x-y|}\cdot\nabla\right)f^\mu(y)dS_{\tau}(y)\bigg).
\end{align}

\begin{thm}
For $(f^\mu(x),g^\mu(x))\in C^2(\mathbb  H^3)\times C^1(\mathbb H^3)$ the limit of the wave propagator exists,
\begin{equation}
\lim_{\tau_0\to  0^+}  W(0,\tau)(f^\mu,g^\mu)=(A^\mu(\tau,x),\partial_\tau A^\mu(\tau,x)),
\end{equation}
it depends on $f^\mu(x)$ as well as $g^\mu(x)$, and  satisfies
\begin{equation}
\lim_{\tau\to 0^+} A^\mu(\tau,x)=f^\mu(x),  \quad \lim_{\tau\to 0+} \partial_\tau A^\mu(\tau,x)=g^\mu(x).
\end{equation}
Thus, the expression gives a solution to the wave equation over the full half-line\\ $\tau\in (0,+\infty)$, with initial data $(A^\mu(0,x),\partial_\tau A^\mu(0,x))=(f^\mu(x),g^\mu(x))$ given at the singular time  $\tau=0$.
\end{thm}
\begin{proof}
Consider expression (4.7) obtained by taking the limit as $\tau_0\to 0$. We find that
\[\lim_{\tau\to 0^+}A^\mu(\tau,x)=\lim_{\tau\to 0^+}M_{f}(\tau,x)=f^\mu(x)\]
for continuous initial data $f_\mu(x)$. In addition, after differentiating we can express it as
\[\partial_\tau A^\mu(\tau,x)=\sinh(\tau)M_f+2\cosh(\tau)\partial_\tau M_f+\sinh(\tau)\partial^2_\tau M_f+\cosh(\tau)M_g+\sinh\partial_\tau M_g\]
from which we obtain as a limit
\[\lim_{\tau\to 0^+}\partial_\tau A^\mu(\tau,x)=\lim_{\tau\to 0^+}M_g(\tau,x)=g^\mu(x)\]
where the terms involving partial derivatives vanish since $M_f$ and $M_g$ are even in $\tau$.
\end{proof}
Similar  to the case $K=0$, we see that the Cauchy problem is indeed well posed for $\tau_0=0$. The solution exists for the initial value problem consisting of initial data $(f^\mu(x),g^\mu(x))$ posed at $\tau_0$, and propagated to future (past) times $\tau$.
\section{Concluding Remarks}
The calculations in this paper give significance  to the Maxwell propagator applied  to data $(f^\mu(x),g^\mu(x))$ posed at $\tau=0$ for curvature $K=0$ as  well as $K=-1$. It is interesting that the Cauchy problem for the electromagnetic wave equation in these space-times is indeed well posed for $\tau_0=0$, where solutions exist for initial data given by vector like functions $(f^\mu(x),g^\mu(x))$ posed at $\tau_0$ and propagated to future (past) times $\tau$. This is in stark contrast to the scalar wave equation in which the full Cauchy problem for both cases of curvature is not well posed for $\tau_0=0$. Solutions of the general Cauchy problem posed at $\tau_0>0$ and propagated to times $0<\tau<\tau_0$ become singular as $\tau\to 0^+$ [1]. However, there remains a full function space of initial data, depending upon one scalar function $g(x)$, for which the solution exists for the initial value problem consisting of data $(g(x),0)$ posed at $\tau_0$,  and propagated to future (past) times $\tau$. The result is a  well defined solution for $\tau>0$ emanating from the space-time singularity  at $\tau_0=t_0=0$. Under reflection $\tau\to-\tau$, the expression is also a solution to the wave equation on FRW space-times for $\tau<0$. We thus obtain a global solution whose evolution is continuous across the space-time singularity at $\tau=0$ (the image of $t=0$ as the transformed time variable  in both cases $K=0$ and $K=-1$ behaves asymptotically as $t\sim\tau^3/3$ for $\tau\to 0$ [1]). 
This is  interesting as the result  shows that light can propagate continuously through a  big bang singularity of the type in our universe and is analogous to that obtained by Abbasi and Craig [1] for the scalar wave equation. Note however, that there is a $\textit{possible}$ loss of regularity of the solution. Namely, if $f_\mu\in C^2$ and $g_\mu\in C^1$ initially, then the solution expression guarantee's only that  $f_\mu\in C^1$ and $g_\mu\in  C$ at  a later time. This can be attributed to focussing effects present when $n>1$ (John F.) [7].\\\\
Secondly, given the explicit nature of the solution one can $\textit{read off}$ the decay rate. For the case of curvature $K=0$ the time  variable $t$ is related to the conformal time variable $\tau$ by $t=\tau^3/3$. It follows that the solution decays as $\mathcal{O}(t^{-1/3})$. This is in contrast to the decay rate obtained for the scalar wave equation on FRW space-times for curvature $K=0$, of which was found to be $\mathcal  O(t^{-1})$ (identical to that of Minkowski space-time). The results is that for $K=0$ the electromagnetic wave equation (being a vector wave equation) obeys a slower decay rate than that of its scalar counterpart. For the case of curvature $K=-1$ there is not a clean inversion formulae as the time variable $t$ is related to the conformal time variable $\tau$ by the relation $t=\sinh(\tau)-\tau$. In order to determine the decay rate one must look at the asymptotic's of the solution. Doing  so, the solution decays as $\mathcal{O}(t^{-1})$ as  the solution expression is bounded  for large  $\tau$  by $e^{-\tau}$. This decay rate is once again found to be slower than that of the scalar wave equation, which was calculated to be $\mathcal{O}(t^{-2})$ [1]. It is interesting to note that the decay rate for the electromagnetic wave equation in FRW space-times of curvature $K=-1$ is identical to that of the electromagnetic wave equation in Minkowski space-time and to that of the scalar wave equation in FRW space-times of curvature $K=0$.\\\\
Finally, noting that both Huygens' principle of finite propagation speed as well as Huygens' sharp principle (that solutions lie on the boundary of the light cone) are satisfied in the underlying product metrics considered, it follows that both principles are satisfied on the FRW space-times considered here [6]. Conversely, although Huygens' principle of finite propagation speed is indeed satisfied for the scalar wave equation on these space-times, Huygens sharp principle is not. Rather, in the case of $K=0$ the solution at a point is seen to persist indefinitely, spatially constant with value related to the average value of the initial data and with asymptotically diminishing magnitude in time. In the case of $K=-1$ the solution in the interior of the light cone is not locally spatially constant, but is dictated by a kernel which is dependent upon the geodesic radius $r$ [1].

\section*{Acknowledgments}
This work is in part supported by the James Stewart Research Award. It is a pleasure to thank Dr. M. Wang and  Dr. D. Pelinovsky for their time and commitment in helping complete my senior thesis project, of which is distilled into this paper. I'd also like to thank Dr. N. Kamran whose comments and guidance were indispensable in making this paper happen. A special thank you to my supervisor and mentor Dr. W. Craig who sadly passed away in January 2019. His  ideas and knowledge live on in the many students and colleagues fortunate enough to have experienced him.

\end{document}